\documentclass[a4paper,12pt]{article}
\usepackage[utf8]{inputenc}

\usepackage{graphicx}
\usepackage{geometry}
\usepackage{amsthm}
\usepackage{amssymb}
\usepackage{amsmath}

\usepackage[colorlinks]{hyperref}
\hypersetup{
    allcolors = black
}

\usepackage{xcolor}
\usepackage{enumerate}
\usepackage{listings}
\usepackage{microtype}
\usepackage{pifont}

\graphicspath{{figs/}}

\newtheorem{theorem}{Theorem}

\newtheorem{lemma}[theorem]{Lemma}
\newtheorem{proposition}[theorem]{Proposition}
\newtheorem{corollary}[theorem]{Corollary}
\newtheorem{problem}{Problem}
\newtheorem{conjecture}{Conjecture}

\definecolor{limeGreen}{RGB}{50,205,50}

\def\inst#1{$^{#1}$}

\date{}

\title{On Helly numbers of exponential lattices\thanks{G. Ambrus was partially supported by ERC Advanced Grant "GeoScape" no.  882971,  by the Hungarian National Research grant no. NKFIH KKP-133819,  and by project no. TKP2021-NVA-09, which has been implemented with the support provided by the Ministry of Innovation and Technology of Hungary from the National
Research, Development and Innovation Fund, financed under the
TKP2021-NVA funding scheme.
M. Balko was supported by the grant no. 21/32817S of the Czech Science Foundation (GA\v{C}R) and by the Center for Foundations of Modern Computer Science (Charles University project UNCE/SCI/004).
N. Frankl was partially supported by ERC Advanced Grant "GeoScape". 
A. Jung was supported by the Rényi Doctoral Fellowship of the Rényi Institute.
M. Nasz{\'o}di was supported by the J{\'a}nos Bolyai Scholarship of the Hungarian Academy of Sciences.
This article is part of a project that has received funding from the European Research Council (ERC) under the European Union's Horizon 2020 research and innovation programme (grant agreement No 810115).}
\author{Gergely Ambrus\inst{1}
\and
Martin Balko\inst{2}
\and
N\'{o}ra Frankl\inst{3}
\and
Attila Jung\inst{4}
\and
M\'{a}rton Nasz\'{o}di\inst{5}
}
}

\begin{document}

\maketitle

\begin{center}
{\footnotesize
\inst{1} 
Dpeartment of Geometry, Bolyai Institute, University of Szeged, Hungary, and\\
Alfr\'{e}d R\'{e}nyi Institute of Mathematics, Hungary
 \\
\texttt{ambrus@renyi.hu}\\
\inst{2} 
Department of Applied Mathematics, \\
Faculty of Mathematics and Physics, Charles University, Czech Republic \\
\texttt{balko@kam.mff.cuni.cz}\\
\inst{3} 
School of Mathematics and Statistics, The Open University, UK, and Alfr\'{e}d R\'{e}nyi Institute of Mathematics, Hungary \\
\texttt{nora.frankl@open.ac.uk}\\
\inst{4} 
Institute of Mathematics, ELTE E\"{o}tv\"{o}s Lor\'{a}nd University, Hungary \\
\texttt{jungattila@gmail.com}\\
\inst{5} 
Alfr\'{e}d R\'{e}nyi Institute of Mathematics and
Department of Geometry, E\"{o}tv\"{o}s Lor\'{a}nd University, Hungary.\\
\texttt{marton.naszodi@math.elte.hu}\\
}
\end{center}

\begin{abstract}
Given a  set $S \subseteq \mathbb{R}^2$, define the \emph{Helly number of $S$}, denoted by $H(S)$, as the smallest positive integer $N$, if it exists, for which the following statement is true: for any finite family $\mathcal{F}$ of convex sets in~$\mathbb{R}^2$ such that the intersection of any $N$ or fewer members of~$\mathcal{F}$ contains at least one point of $S$, there is a point of $S$ common to all members of $\mathcal{F}$.

We prove that the Helly numbers of \emph{exponential lattices} $\{\alpha^n \colon n \in \mathbb{N}_0\}^2$ are finite for every $\alpha>1$ and we determine their exact values in some instances.
In particular, we obtain $H(\{2^n \colon n \in \mathbb{N}_0\}^2)=5$, solving a problem posed by Dillon (2021).

For real numbers $\alpha, \beta > 1$, we also fully characterize exponential lattices $L(\alpha,\beta) = \{\alpha^n \colon n \in \mathbb{N}_0\} \times \{\beta^n \colon n \in \mathbb{N}_0\}$ with finite Helly numbers by showing that $H(L(\alpha,\beta))$ is finite if and only if $\log_\alpha(\beta)$ is rational.
\end{abstract}

\section{Introduction}

\emph{Helly's theorem}~\cite{helly23} is one of the most classical results in combinatorial geometry.
It states that, for each $d \in \mathbb{N}$, if the intersection of any $d + 1$ or fewer members of a finite family~$\mathcal{F}$ of convex sets in $\mathbb{R}^d$ is nonempty, then the entire family $\mathcal{F}$ has nonempty intersection.
There have been numerous variants and generalizations of this famous result; see~\cite{amDeSo17,holWeg17} for example.
One active direction of this research with rich connections to the theory of optimization, in particular to integer programming and LP-type problems~\cite{amDeSo17,confSum16}, is the study of variants of Helly's theorem with coordinate restrictions, which is captured by the following definition.

Let $d$ be a positive integer.
The \emph{Helly number} of a set $S \subseteq \mathbb{R}^d$, denoted by~$H(S)$, is the smallest positive integer $N$, if it exists, such that the following statement is true for every finite family $\mathcal{F}$ of convex sets in $\mathbb{R}^d$: if the intersection of any $N$ or fewer members of $\mathcal{F}$ contains at least one point of $S$, then $\bigcap \mathcal{F}$ contains at least one point of $S$. 
If no such number $N$ exists, then we write $H(S) = \infty$.
Helly's theorem in this language can be restated as $H(\mathbb{R}^d)=d+1$.

A classical result of this sort is \emph{Doignon's theorem}~\cite{doignon73} where the set $S$ is the integer lattice~$\mathbb{Z}^d$.
This result, which was also independently discovered by Bell~\cite{bell76} and by Scarf~\cite{scarf77}, states that $H(\mathbb{Z}^d) \leq 2^d$.
This is tight as for $Q = \{0,1\}^d$ the intersection of any $2^d - 1$ sets in the family $\{{\rm conv}(Q \setminus \{x\}) \colon x \in Q\}$ contains a lattice point, but the intersection of all $2^d$ sets  does not.

The theory of Helly numbers of general sets is developing quickly and there are many result of this kind~\cite{amDeSo17,holWeg17}.
For example, De Loera, La Haye, Oliveros, and Rold\'{a}n-Pensado~\cite{deHayeOliRol17} and De Loera, La Haye, Rolnick, and Sober\'{o}n~\cite{deHaRoSo17} studied the Helly numbers of differences of lattices and Garber~\cite{gar17} considered Hely numbers of crystals or cut-and-project sets.

The Helly number of a set $S$ is closely related to the maximum size of a set that is empty in $S$.
A subset $X \subseteq S$ is \emph{intersect-empty} if $\left(\bigcap_{x \in X}{\rm conv}(X \setminus \{x\}) \right) \cap S = \emptyset$.
A convex polytope $P$ with vertices in $S$ is \emph{empty in $S$} if $P$ does not contain any points of $S$ other than its vertices.
In particular, an empty polytope does not contain points of $S$ in the interior of its edges.
For a discrete set $S$, we use $h(S)$ to denote the maximum number of vertices of an empty polytope in $S$.
If there are empty polytopes in $S$ with arbitrarily large number of vertices, then we write $h(S)=\infty$.

The following result by Hoffman~\cite{hoffman79} (which was essentially already proved by Doignon~\cite{doignon73}) shows the close connection between intersect-empty sets and empty polygons in $S$ and the Helly numbers of $S$; see also~\cite{agmpsw17}.

\begin{proposition}[\cite{hoffman79}]
\label{prop-hoffman}
If $S \subseteq \mathbb{R}^d$, then $H(S)$ is equal to the maximum cardinality of an intersect-empty set in $S$. 
If $S$ is discrete, then $H(S)=h(S)$.
\end{proposition}

Since all the sets $S$ studied in this paper are discrete, we state all of our results using $h(\alpha)$ but, due to Proposition~\ref{prop-hoffman}, our results apply to $H(\alpha)$ as well.

Very recently, Dillon~\cite{dillon21} proved that the Helly number of a set $S$ is infinite if $S$ belongs to a certain collection of \emph{product sets}, which are sets of the form $S=A^d$ with a certain kind of discrete set $A \subseteq \mathbb{R}$.
His result shows, for example, that whenever $p$ is a polynomial of degree at least 2 and $d \geq 2$, then $h(\{p(n) \colon n \in \mathbb{N}_0\}^d) = \infty$.
However, there are sets for which Dillon's method gives no information, for example $\{2^n \colon n \in \mathbb{N}_0\}^2$.
Thus, Dillon~\cite{dillon21} posed the following question, which motivated our research.

\begin{problem}[Dillon, \cite{dillon21}]
\label{prob-Dillon}
What is $h(\{2^n \colon n \in \mathbb{N}_0\}^2)$?
\end{problem}

In this paper, we study the Helly numbers of \emph{exponential lattices $L(\alpha)$} and $L(\alpha, \beta)$ in the plane where $L(\alpha)=\{\alpha^n \colon n \in \mathbb{N}_0\}^2$ and $L(\alpha, \beta) = \{\alpha^n \colon n \in \mathbb{N}_0\} \times \{\beta^n \colon n \in \mathbb{N}_0\}$ for real numbers $\alpha, \beta >1$.
In particular, we prove that Helly numbers of exponential lattices $L(\alpha)$ are finite and we provide several estimates that give exact values for $\alpha$ sufficiently large, solving Problem~\ref{prob-Dillon}. We also show that Helly numbers of exponential lattices $L(\alpha, \beta)$ are finite if and only if $\log_\alpha (\beta)$ is rational.
Finally, we introduce some new open problems, for example, it is not even known whether the Helly numbers of the sets $\{\alpha^n \colon n\in\mathbb{N}_0\}^d$ with $d>2$ are finite.

\section{Our results}

For a real number $\alpha >1$ and the exponential lattice $L(\alpha)=\{\alpha^n \colon n \in \mathbb{N}_0\}^2$, we abbreviate $h(L(\alpha))$ by $h(\alpha)$.

As our first result, we provide finite bounds on the numbers $h(\alpha)$ for any $\alpha>1$.
The upper bounds are getting smaller as $\alpha$ increases and reach their minimum at $\alpha=2$.

\begin{theorem}
\label{thm-upperBound}
For every real $\alpha > 1$, the maximum number of vertices of an empty polygon in $L(\alpha)$ is finite.
More precisely, we have $h(\alpha) \leq 5$ for every $\alpha \geq 2$, $h(\alpha) \leq 7$ for every $\alpha \in [\frac{1+\sqrt{5}}{2},2)$, and 
\[h(\alpha) \leq 3\left\lceil\log_\alpha\left(\frac{\alpha}{\alpha-1}\right)\right\rceil+3\]
for every $\alpha \in (1,\frac{1+\sqrt{5}}{2})$.
\end{theorem}

We note that if $\alpha = 1+\frac{1}{x}$ for $x \in (0,\infty)$, then the bound from Theorem~2 becomes $h(1+\frac{1}{x}) \leq O(x\log_2(x))$.
Moreover, we show that the breaking points of $\alpha$ for our upper bounds are determined by certain polynomial equations; see Section~\ref{sec-upperBound}.

We also consider the lower bounds on $h(\alpha)$ and provide the following estimate.

\begin{theorem}
\label{thm-lowerBound}
We have $h(\alpha) \geq 5$ for every $\alpha \geq 2$ and $h(\alpha) \geq 7$ for every $\alpha \in \left[\frac{1+\sqrt{5}}{2},2\right)$.
For every $\alpha \in \left(1,\frac{1+\sqrt{5}}{2}\right)$, we have
\[h(\alpha) \geq \left\lfloor\sqrt{\frac{1}{\alpha - 1}}\right\rfloor.\]
\end{theorem}

If $\alpha = 1+\frac{1}{x}$ where $x \in (0,\infty)$, then the lower bound from Theorem~\ref{thm-lowerBound} becomes $h(1+\frac{1}{x}) \geq \lfloor\sqrt{x}\rfloor$.
So with decreasing $\alpha$, the parameter $h(\alpha)$ indeed grows to infinity.

By combining Theorems~\ref{thm-upperBound} and~\ref{thm-lowerBound}, we get the precise value of the Helly numbers of $L(\alpha)$ with $\alpha \geq (1+\sqrt{5})/2$.
In particular, for $\alpha=2$, we obtain a solution to Problem~\ref{prob-Dillon}.

\begin{corollary}
\label{cor-upperBound}
We have $h(\alpha) = 5$ for every $\alpha \geq 2$ and $h(\alpha) = 7$ for every $\alpha \in [\frac{1+\sqrt{5}}{2},2)$.
\end{corollary}

We prove the following result which shows that even a slight perturbation of $S$ can affect the value $h(S)$ drastically.
We use the \emph{Fibonacci numbers} $(F_n)_{n \in \mathbb{N}_0}$, which are defined as $F_0=1,F_1=1$ and $F_n=F_{n-1}+F_{n-2}$ for every integer $n \geq 2$.

\begin{proposition}
\label{prop-fibonacci}
We have $h(\{F_n \colon n \in \mathbb{N}_0\}^2) = \infty$.
\end{proposition}

We recall that 
$F_n = \frac{\varphi^{n+1} - \psi^{n+1}}{\sqrt{5}}$ for every $n \in \mathbb{N}_0$, where  $\varphi = \frac{1+\sqrt{5}}{2}$ is the \emph{golden ratio} and $\psi = \frac{1-\sqrt{5}}{2}=1-\varphi$ is its conjugate.
Since $\psi < 1$, this formula shows that the points of $\{F_n \colon n \in \mathbb{N}_0\}^2$ are approaching the points of the scaled exponential lattice $\frac{\varphi}{\sqrt{5}} \cdot L(\varphi) = \{\frac{\varphi}{\sqrt{5}}\cdot \varphi^n \colon n\in \mathbb{N}_0\}^2$.
Thus, Proposition~\ref{prop-fibonacci} is in sharp contrast with the fact that $h(\frac{\varphi}{\sqrt{5}} \cdot L(\varphi))=h(\varphi)\leq 7$, which follows from Theorem~\ref{thm-upperBound} and from the fact that affine transformations of any set $S \subseteq \mathbb{R}^d$ do not change $h(S)$.
We also note Dillon's method~\cite{dillon21} does not imply $h(\{F_n \colon n \in \mathbb{N}_0\}^2) = \infty$.

We also consider the more general case of exponential lattices where the rows and the columns might use different bases.
For real numbers $\alpha>1$ and $\beta>1$, let $L(\alpha,\beta)$ be the set $\{\alpha^n \colon n \in \mathbb{N}_0\}\times\{\beta^n \colon n \in \mathbb{N}_0\}$.
Note that $L(\alpha) = L(\alpha,\alpha)$ for every $\alpha>1$.

As our last main result, we fully characterize exponential lattices $L(\alpha,\beta)$ with finite Helly numbers $h(L(\alpha,\beta))$, settling the question of finiteness of Helly numbers of planar exponential lattices completely.

\begin{theorem}
\label{thm-nondiagonal}
Let $\alpha>1$ and $\beta>1$ be real numbers.
Then, $h(L(\alpha,\beta))$ is finite if and only if $\log_\alpha(\beta)$ is a rational number.

Moreover, if $\log_\alpha(\beta) \in \mathbb{Q}$, that is, $\beta=\alpha^{p/q}$ for some $p,q \in \mathbb{N}$, then
\[\left\lfloor \frac{1}{pq} \left\lfloor\sqrt{\frac{1}{\alpha^{1/q} - 1} } \right \rfloor  \right\rfloor \leq h(L(\alpha,\beta)) \leq pq\cdot h(\alpha^p).\]
\end{theorem}

The proof of Theorem~\ref{thm-nondiagonal} is based on Theorem~\ref{thm-upperBound} and on the theory of continued fractions and Diophantine approximations.

\subsubsection*{Open problems}
First, it is natural to try to close the gap between the upper bound from Theorem~\ref{thm-upperBound} and the lower bound from Theorem~\ref{thm-lowerBound} and potentially obtain new precise values of $h(\alpha)$.

Second, we considered only the exponential lattice in the plane, but it would be interesting to obtain some estimates on the Helly numbers of exponential lattices $\{\alpha^n \colon n \in \mathbb{N}_0\}^d$ in dimension $d > 2$.
In particular, are these numbers finite?

We also mention the following conjecture of De Loera, La Haye, Oliveros, and Rold\'{a}n-Pensado~\cite{deHayeOliRol17}, which inspired the research of Dillon~\cite{dillon21}.

\begin{conjecture}[\cite{deHayeOliRol17}]
\label{conj-primes}
If $\mathcal{P}$ is the set of prime numbers, then $h(\mathcal{P}^2)=\infty$.
\end{conjecture}
\noindent
Using computer search, Summers~\cite{summers15} showed that $h(\mathcal{P}^2) \geq 14$.

\section{Proof of Theorem~\ref{thm-upperBound}}
\label{sec-upperBound}

Here, we prove Theorem~\ref{thm-upperBound} by showing that the number $h(\alpha)$ is finite for every $\alpha > 1$.
This follows from the upper bounds
$h(\alpha) \leq 5$ for $\alpha \geq 2$, $h(\alpha) \leq 7$ for every $\alpha \geq [\frac{1+\sqrt{5}}{2},2)$, and
\[h(\alpha) \leq 3\left\lceil\log_\alpha\left(\frac{\alpha}{\alpha-1}\right)\right\rceil+3\] 
for any $\alpha \in (1,\frac{1+\sqrt{5}}{2})$.

We start by introducing some auxiliary definitions and notation.
Let $\alpha > 1$ be a real number and consider the exponential lattice $L(\alpha)$.
For $i \in \mathbb{N}_0$, the \emph{$i$th column} of $L(\alpha)$ is the set $\{(\alpha^i,\alpha^n )\colon n \in \mathbb{N}_0\}$.
Analogously, the \emph{$i$th row} of $L(\alpha)$ is the set $\{(\alpha^n, \alpha^i)\colon n \in \mathbb{N}_0\}$.

For a point $p$ in the plane, we write $x(p)$ and $y(p)$ for the $x$- and $y$-coordinates of $p$, respectively.
Let $P$ be an empty convex polygon in $L(\alpha)$.
Let $e$ be an edge of $P$ connecting vertices $u$ and $v$ where $x(u)<x(v)$ or $y(u)<y(v)$ if $x(u)=x(v)$.
We use $\overline{e}$ to denote the line determined by $e$ and oriented from $u$ to $v$.
The \emph{slope of $e$} with $x(u)<x(v)$ is the slope of $\overline{e}$, that is, $\frac{y(v)-y(u)}{x(v)-x(u)}$.

We distinguish four types of edges of $P$; see part~(a) of Figure~\ref{fig-types}.
Roughly speaking, the type of an edge is exactly the quadrant where the normal vector to this edge points to (up to the boundaries of the quadrants).
First, assume $x(u) \neq x(v)$ and $y(u) \neq y(v)$.
We say that $e$ is of \emph{type I} if the slope of $e$ is negative and $P$ lies to the right of $\overline{e}$.
Similarly, $e$ is of \emph{type II} if the slope of $e$ is positive and $P$ lies to the right of $\overline{e}$.
An edge $e$ has \emph{type III} if the slope of $e$ is negative and $P$ lies to the left of $\overline{e}$.
Finally, \emph{type IV} is for $e$ with positive slope and with $P$ lying to the left of $\overline{e}$.
It remains to deal with horizontal and vertical edges of $P$.
A horizontal edge $e$ is of type II if $P$ lies below $\overline{e}$ and is of type III otherwise.
Similarly, a vertical edge $e$ is of type IV if $P$ lies to the left of $\overline{e}$ and is of type III otherwise.

\begin{figure}[ht]
    \centering
    \includegraphics{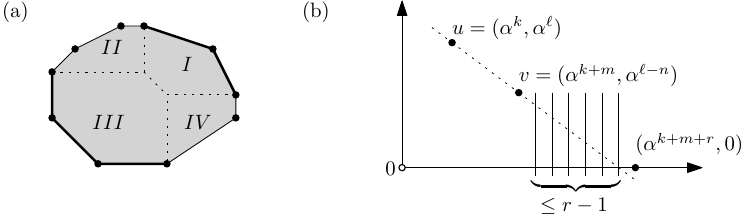}
    \caption{(a) The four types of edges of a convex polygon. (b) An illustration of the proof of Lemma~\ref{lem-typeI}.}
    \label{fig-types}
\end{figure}

Note that each edge of $P$ has exactly one type and that the types partition the edges of $P$ into four convex chains.
We first provide an upper bound on the number of edges of those chains of $P$ and then derive the bound on the total number of edges of $P$ by summing the four bounds.
We start by estimating the number of edges of $P$ of type I.

\begin{lemma}
\label{lem-typeI}
The polygon $P$ has at most $\left\lceil \log_{\alpha}\left(\frac{\alpha}{\alpha-1}\right) \right\rceil$ edges of type I.
\end{lemma}
\begin{proof}
First, let $r = \left\lceil\log_{\alpha}\left(\frac{\alpha}{\alpha-1}\right)\right\rceil$ and note that $r\geq 1$ as $\alpha>1$.
Let $e$ be the left-most edge of $P$ of type I and let $u$ and $v$ be the vertices of $e$.
Since $e$ is of type I, we have $u=(\alpha^k,\alpha^\ell)$ and $v=(\alpha^{k+m},\alpha^{\ell-n})$ for some positive integers $k$, $\ell$, $m$, and $n$.

We will show that the point $(\alpha^{k+m+r},0)$ lies above the line $\overline{e}$.
Since there are at most $r-1$ columns of $L(\alpha)$ between the vertical line containing $v$ and the vertical line containing $(\alpha^{k+m+r},0)$ and the point $(\alpha^{k+m+r},0)$ is below the lowest row of $L(\alpha)$, it then follows that there are at most $r$ edges of $P$ of type I; see part~(b) of Figure~\ref{fig-types}.

Since the line $\overline{e}$ contains $u$ and $v$, we see that
\[\overline{e} = \{(x,y) \in \mathbb{R}^2 \colon (\alpha^\ell - \alpha^{\ell-n})x + (\alpha^{k+m} - \alpha^k)y = \alpha^{k+\ell+m} - \alpha^{k+\ell-n}\}.\]
It suffices to check that by substituting the coordinates of the point $(\alpha^{k+m+r},0)$ into the equation of the line $\overline{e}$ gives a left side that is at least $\alpha^{k+\ell+m} - \alpha^{k+\ell-n}$.
The left side equals $\alpha^{k+\ell+m+r} - \alpha^{k+\ell+m-n+r}$ and thus we want 
\[\alpha^{k+\ell+m+r} - \alpha^{k+\ell+m-n+r} \geq \alpha^{k+\ell+m} - \alpha^{k+\ell-n}.\]
By dividing both sides by $\alpha^{k+\ell}$ and by rearranging the terms, we can rewrite this expression as
\[\alpha^{-n}(1 - \alpha^{m+r}) \geq \alpha^{m}-\alpha^{m+r}.\]
Since $m,r>0$ and $\alpha>1$, we get $(1-\alpha^{m+r})<0$ and thus the left side is increasing as $n$ increases, so we can assume $n=1$, leading to
\[\alpha^{-1} - \alpha^{m+r-1} \geq \alpha^m-\alpha^{m+r}.\]
We can again rearrange the inequality as
\[\alpha^r -\alpha^{r-1} - 1 \geq -\alpha^{-1-m},\]
where the right side is negative and approaches 0 as $m$ tends to infinity, so we can replace it by 0, obtaining
\[\alpha^r -\alpha^{r-1} \geq 1.\]
This inequality is satisfied by our choice of $r$.
\end{proof}

We now estimate the number of edges of $P$ that are of type III.

\begin{lemma}
\label{lem-typeIII}
The polygon $P$ has at most $2\lceil\log_\alpha\left(\frac{\alpha+1}{\alpha}\right)\rceil+1$ edges of type III for $1< \alpha < 2$ and at most $2$ such edges for $\alpha \geq 2$.
\end{lemma} 
\begin{proof}
Let $t = \lceil\log_\alpha\left(\frac{\alpha+1}{\alpha}\right)\rceil$ and $s=t+1$ for $\alpha \in (1,2)$ and $t=1=s$ for $\alpha \geq 2$.
Suppose for contradiction that there are $s+t+1$ edges of $P$ of type III.
Let $v_1,\dots,v_{s+t+2}$ be the vertices of the convex chain that is formed by edges of $P$ of type III.
We use $Q$ to denote the convex polygon with vertices $v_1,\dots,v_{s+t+2}$.
Note that $Q$ is empty in $L(\alpha)$ as $P$ is empty and $Q\subseteq P$.

Let $v'$ be the point $(x(v_{s+2}),\alpha \cdot y(v_{s+2}))$, that is, $v'$ is the point of $L(\alpha)$ that lies just above $v_{s+2}$; see part~(a) of Figure~\ref{fig-typeIII}.
We will show that the point $v'$ lies below the line $\overline{v_1v_{s+t+2}}$.
Since $v'$ lies in the same column of $L(\alpha)$ as $v_{s+2}$, this then implies that $v'$ lies in the interior of $Q$, contradicting the fact that $Q$ is empty in $L(\alpha)$.

\begin{figure}[ht]
    \centering
    \includegraphics{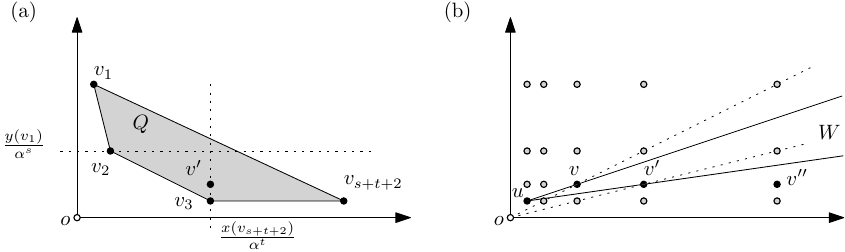}
    \caption{(a) An illustration of the proof of Lemma~\ref{lem-typeIII} for $s=1=t$. (b) An illustration of Lemma~\ref{lem-wedge}.}
    \label{fig-typeIII}
\end{figure}

Note that $x(v') \leq \frac{x(v_{s+t+2})}{\alpha^t}$ and $y(v') \leq \frac{y(v_1)}{\alpha^s}$ as all edges $v_iv_{i+1}$ are of type III and thus the $x$- and $y$-coordinates decrease by a multiplicative factor at least $\alpha$ for each such edge.
Since the only vertical edge might be $v_1v_2$ and the only horizontal edge might be $v_{s+t+1}v_{s+t+2}$, the $x$- or $y$-coordinates indeed decrease by the  factor at least $\alpha$ at each step.

Let $v_1=(\alpha^k,\alpha^\ell)$ and $v_{s+t+2}=(\alpha^{k+m},\alpha^{\ell-n})$ for some positive integers $k,\ell,m,n$.
Note that $m,n \geq s+t$.
The line determined by $v_1$ and $v_{s+t+2}$ is then
\[\{(x,y) \in \mathbb{R}^2 \colon (\alpha^\ell - \alpha^{\ell-n})x + (\alpha^{k+m} - \alpha^k)y = \alpha^{k+\ell+m} - \alpha^{k+\ell-n}\}.\]
Since $x(v') \leq \frac{x(v_{s+t+2})}{\alpha^t}$ and $y(v') \leq \frac{y(v_1)}{\alpha^s}$, it suffices to check
\[(\alpha^\ell - \alpha^{\ell-n})\frac{\alpha^{k+m}}{\alpha^t} + (\alpha^{k+m} - \alpha^k)\frac{\alpha^\ell}{\alpha^s} < \alpha^{k+\ell+m} - \alpha^{k+\ell-n}.\]
After dividing by $\alpha^{k+\ell +m}$, this can be rewritten as 
\[\alpha^{-t} + \alpha^{-s} < 1 - \alpha^{-m-n} + \alpha^{-t-n} + \alpha^{-s-m}.\]
Since $m,n \geq s+t$, the right hand side is decreasing with increasing $m$ and $n$ and thus we only need to prove
\[\alpha^{-s} + \alpha^{-t} \leq 1.\]
If $\alpha \geq 2$, then $s=1=t$ and this inequality becomes $2/\alpha \leq 1$, which is true.
If $\alpha \in (1,2)$, then $s=t+1$ and the inequality becomes $1+1/\alpha \leq \alpha^t$, which is also true by our choice of~$t$.
\end{proof}

It remains to bound the number of edges of $P$ that are  of types II and IV.
Observe that if we switch the $x$- and $y$- coordinates of $P$, then edges of type II become edges of type IV and vice versa.
Since the exponential lattice $L(\alpha)$ is symmetric with respect to the line $x=y$, we see that  it suffices to estimate the number of edges of type II.
To do so, we use the following auxiliary result.

\begin{lemma}
\label{lem-wedge}
Let $u$ be a point of $L(\alpha)$ and let $v$ and $v'$ be two points of $L(\alpha)$ that are consecutive in a row $R$ of $L(\alpha)$ that lies above the row containing $u$; see part~(b) of Figure~\ref{fig-typeIII}.
If $v$ and $v'$ lie to the right of $u$, then all points of $L(\alpha)$ that lie above $R$ in the interior of the angle $W$ spanned by the rays $\overline{uv}$ and $\overline{uv'}$ lie on at most $\left\lceil \log_\alpha(\frac{\alpha}{\alpha-1})\right\rceil$ lines containing the origin. 
\end{lemma}
\begin{proof}
Similarly as in Lemma~\ref{lem-typeI}, we set $r = \left\lceil\log_{\alpha}\left(\frac{\alpha}{\alpha-1}\right)\right\rceil$ and note that $r \geq 1$.
We can assume without loss of generality that $u=(1,1)$ as otherwise it suffices to scale the points of $L(\alpha)$ with an affine transformation.
Since $v$ and $v'$ are consecutive on~$R$, they both lie in the same closed halfplane determined by the line $x=y$. 
We first assume that the point $v$ lies below or on the line $x=y$.

Let $o$ be the origin and consider the lines $\overline{ov}$ and $\overline{ov'}$.
Then, the part of the line $\overline{ov}$ above the row $R$ is (not necessarily strictly) above $\overline{uv}$; see part~(b) of Figure~\ref{fig-typeIII}.
Similarly, the part of the line $\overline{ov'}$ above $R$ is above $\overline{uv'}$.
It follows that only points of $L(\alpha)$ that lie on a line $\overline{ow}$, where $w$ is a point of $L(\alpha)$ to the right of $v$ on~$R$, can lie in the interior of $W$.

Let $v''$ be the point $(\alpha^r \cdot x(v'),y(v'))$, that is, $v''$ is the point of $L(\alpha)$ that lies  $r$ columns to the right of $v'$ on $R$,
We will show that the part of the line $\overline{ov''}$ above~$R$ lies below $\overline{uv'}$.
This will conclude the proof as all points of $L(\alpha)$ that lie in the interior of $W$ above $R$ have to then lie on one of the $r$ lines $\overline{ow}$ with $w$ lying between $v$ and $v''$ on $R$.

It suffices to compare the slopes of the lines $\overline{ov''}$ and $\overline{uv'}$.
Let $v' = (\alpha^m,\alpha^n)$ for some positive integers $m$ and $n$.
Then, the slope of $\overline{ov''}$ is 
\[\frac{y(v'')-y(o)}{x(v'')-x(o)} = \frac{y(v')}{\alpha^r \cdot x(v')} = \frac{\alpha^n}{\alpha^{m+r}}\]
and the slope of $\overline{uv'}$ equals
\[\frac{y(v')-y(u)}{x(v')-x(u)} = \frac{y(v') - 1}{x(v')-1} = \frac{\alpha^n-1}{\alpha^m-1}.\]
Thus, we want
\[\frac{\alpha^n - 1}{\alpha^m-1} \geq \frac{\alpha^n}{\alpha^{m+r}}.\]
We can rewrite this inequality as
\[\alpha^{m+n+r} - \alpha^{m+r} \geq  \alpha^{n+m}-\alpha^n,\]
which can be further rewritten by dividing both sides with $\alpha^n$ as
\[\alpha^{m+r}(1-\alpha^{-n}) \geq  \alpha^m - 1.\]
The left side is increasing with increasing $n$, so we can assume $n=1$ and, by dividing both sides with $\alpha^m$, we obtain
\[\alpha^r(1-\alpha^{-1}) \geq 1-\alpha^{-m}.\]
Now, the term $\alpha^{-m}$ on the right side approaches $0$ from above with increasing $m$, so we can replace it by $0$ obtaining
\[\alpha^r-\alpha^{r-1} \geq 1.\]
This inequality is satisfied by our choice of $r$.

Now, assume that the point $v$ lies above the line $x=y$.
Then, the proof proceeds analogously as in the previous case.
The part of the line $\overline{ov}$ above the row $R$ is (not necessarily strictly) below $\overline{uv}$.
Similarly, the part of the line $\overline{ov'}$ above $R$ is below $\overline{uv'}$.
Then, only points of $L(\alpha)$ that lie on a line $\overline{ow}$, where $w$ is a point of $L(\alpha)$ to the left of $v$ on~$R$, can lie in the interior of $W$ above $R$.
Considering the point $(\alpha^{-r} \cdot x(v'),y(v'))$ of $L(\alpha)$ that lies  $r$ columns to the left of $v'$ on $R$, we can show with analogous computations as before that the part of the line $\overline{ov''}$ above~$R$ lies above $\overline{uv'}$.
This concludes the proof.
\end{proof}

Now, we can apply Lemma~\ref{lem-wedge} to obtain an upper bound on the number of edges of $P$ of type II.

\begin{lemma}
\label{lem-typeII}
The polygon $P$ has at most $\left \lceil\log_{\alpha}\left(\frac{\alpha}{\alpha-1}\right)\right\rceil+1$ edges of type II. 
\end{lemma}
\begin{proof}
Again, let $r = \left\lceil\log_{\alpha}\left(\frac{\alpha}{\alpha-1}\right)\right\rceil$.
Let $u$ be the leftmost vertex of the convex chain $C$ determined by the edges of $P$ of type II.
Similarly, let $v$ be the second leftmost vertex of~$C$.
Note that since the edge $uv$ is of type II, the vertex $v$ lies in a row $R$ of $L(\alpha)$ above the row containing $u$ and $v$ is also to the right of $u$.
Let $v'$ be the point $(\alpha \cdot  x(v),y(v))$, that is, the point of  $L(\alpha)$ that is to the right of $v$ on~$R$.
Then, by Lemma~\ref{lem-wedge}, all points of $L(\alpha)$ that lie above $R$ and in the interior of the angle $W$ spanned by the rays $\overline{uv}$ and $\overline{uv'}$ lie on at most $r$ lines containing the origin.

Since $P$ is empty in $L(\alpha)$, all vertices of $C$ besides $u$, $v$, and possibly $v'$ lie in $W$ above $R$.
Since all edges of $C$ are of type II, every line determined by the origin and by a point of $L(\alpha)$ from the interior of $W$ contains at most one vertex of $C$.
Note that if $v'$ is a vertex of $C$, then the only vertices of $C$ are $u,v,v'$.
Thus, in total $C$ has at most $r+2$ vertices and therefore at most $r+1$ edges.
\end{proof}

We recall that, by symmetry, the same bound applies for edges of type IV and thus we get the following result.

\begin{corollary}
\label{cor-typeIV}
The polygon $P$ has at most $\left\lceil\log_{\alpha}\left(\frac{\alpha}{\alpha-1}\right)\right\rceil+1$ edges of type IV.\qed
\end{corollary}

Since each edge of $P$ is of one of the types I--IV, it immediately follows from Lemmas~\ref{lem-typeI},~\ref{lem-typeIII},~\ref{lem-typeII}, and from Corollary~\ref{cor-typeIV} that the number of edges of $P$ is at most 
\[3\left\lceil \log_{\alpha}\left(\frac{\alpha}{\alpha-1}\right) \right\rceil + 2 + 2\left\lceil \log_{\alpha}\left(\frac{\alpha+1}{\alpha}\right) \right\rceil +1 \leq 5\left\lceil \log_{\alpha}\left(\frac{\alpha}{\alpha-1}\right) \right\rceil+3,\]
as $\log_x\left(\frac{x}{x-1}\right) \geq \log_{x}\left(\frac{x+1}{x}\right)$ for every $x >1$.
In particular, this gives $h(2) \leq 8$ and $h\left(\frac{1+\sqrt{5}}{2}\right) \leq 13$.
To obtain better bounds that are tight for $\alpha \geq \frac{1+\sqrt{5}}{2}$, we observe that not all types can appear simultaneously.
To show this, we will use one last auxiliary result.

Let $p$ and $q$ be points lying on the same row $R$ of~$L(\alpha)$, each contained in an edge of $P$.
We note that $p$ and $q$ do not need to be distinct and that they can also be interior points of an edge of $P$.
Let $L$ and $L'$ be two lines containing $p$ and $q$, respectively.
If the slopes of $L$ and $L'$ are negative, 
then we call the part of the plane between $L$ and $L'$ below $R$ a \emph{slice of negative slope}; see part~(a) of Figure~\ref{fig-slice}
Analogously, \emph{a slice of positive slope} is the part of the plane between $L$ and $L'$ above $R$ if $L$ and $L'$ 
have positive slope. 

\begin{figure}[ht]
    \centering
    \includegraphics{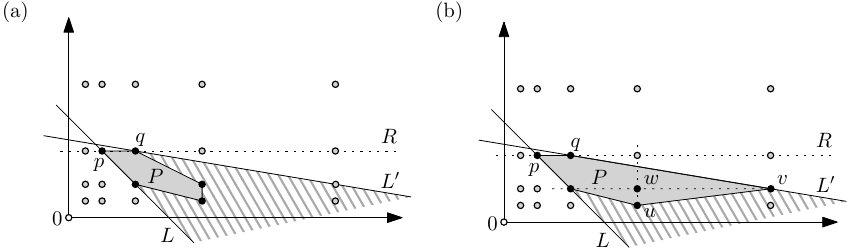}
    \caption{(a) An example of a slice of negative slope. The slice is denoted by dark gray stripes. (b) An illustration of the proof of Lemma~\ref{lem-forbidden} for negative slopes.}
    \label{fig-slice}
\end{figure}

\begin{lemma}
\label{lem-forbidden}
If the empty polygon $P$ is contained in a slice of negative slope, then there is no non-vertical edge of $P$ of type IV.
Similarly, if $P$ is contained in a slice of positive slope, then there is no edge of type I.
\end{lemma}
\begin{proof}
It suffices to prove the statement for slices of negative slope as the proof of the statement for the positive slope is analogous.
Suppose for contradiction that there is a non-vertical edge $uv$ of type IV in a slice of negative slope determined by lines $L$ and $L'$ and points $p$ and $q$ as in the definition of a slice.
Without loss of generality, we assume $x(u)<x(v)$.

Consider the point $w=(x(u),y(v))$ of $L(\alpha)$.
Since $uv$ is non-vertical, we have $w \notin \{u,v\}$.
We claim that $w$ is in the interior of $P$, contradicting the assumption that $P$ is empty in $L(\alpha)$.
Since $uv$ is of type IV, the point $u$ lies below the row containing~$w$.
However, since $p$ is contained in an edge of $P$ and $P$ is in the slice, the boundary of $P$ intersects this row to the left of $w$.
Analogously, $v$ is  to the right of the column containing $w$ and thus the boundary of $P$ intersects this column above $w$.
Then, however, $w$ lies in the interior of $P$.
\end{proof}

Finally, we can now finish the proof of Theorem~\ref{thm-upperBound}.

\begin{proof}[Proof of Theorem~\ref{thm-upperBound}]
First, we observe that if all vertices of $P$ lie on two columns of $L(\alpha)$, then $P$ can have at most four vertices.
So we assume that this is not the case.
Let $u$ be the leftmost vertex of $P$ with the highest $y$-coordinate among all leftmost vertices of~$P$.
Let $e_1$ and $e_2$ be the edges of $P$ incident to $u$.
We denote the other edge of $P$ adjacent to $e_1$ as $e$ and the other edge of $P$ adjacent to $e_2$ as $e'$.
We also use $t_I$, $t_{II}$, $t_{III}$, and $t_{IV}$ to denote the number of edges of $P$ of type I, II, III, and IV, respectively.

First, assume that $e_1$ is vertical. 
If $e_2$ is horizontal, then, since $u$ is the top vertex of~$e_1$ and $P$ is not contained in two columns of $L(\alpha),$ the point $(\alpha \cdot x(u),y(u)/\alpha)$ of  $L(\alpha)$ lies in the interior of $P$, which is impossible as $P$ is empty in $L(\alpha)$.

\begin{figure}[ht]
    \centering
    \includegraphics{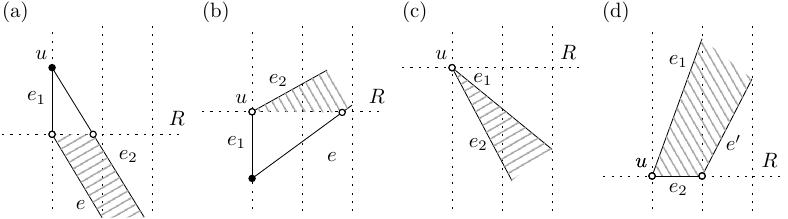}
    \caption{An illustration of the proof of Theorem~\ref{thm-upperBound}. A slice is denoted by grey strips and its points $p$ and $q$ by empty circles. For example, in part~(a), the slice is formed by the region between lines containing the edges $e$ and $e_2$.}
    \label{fig-altogether}
\end{figure}

If $e_1$ is vertical and the slope of $e_2$ is negative, then there is no edge of type~II.
Thus, the edge $e$ intersects the row $R$ of $L(\alpha)$ containing the other vertex of $e_1$ and $\overline{e}$ has negative slope.
Then, the part of $P$ below $R$ is contained in the slice of negative slope determined by $\overline{e_2}$ and $\overline{e}$; see part~(a) of Figure~\ref{fig-altogether}.
By Lemma~\ref{lem-forbidden}, there is no non-vertical edge of type IV in $P$.
By Lemmas~\ref{lem-typeI} and~\ref{lem-typeIII}, the total number of edges of $P$ is thus at most 
\[t_I + t_{III} + 1 \leq \left\lceil\log_\alpha\left(\frac{\alpha}{\alpha-1}\right)\right\rceil+2\left\lceil \log_{\alpha}\left(\frac{\alpha+1}{\alpha}\right) \right\rceil +2\]
for $\alpha \in (1,2)$.
Moreover, for $\alpha \geq 2$ this bound may be reduced by one.

If $e_1$ is vertical and the slope of $e_2$ is positive, then, since $P$ is empty, there is no edge of type III besides $e_1$ as otherwise the point $(\alpha \cdot x(u),y(u))$ of $L(\alpha)$ is in the interior of~$P$.
The edge $e$ intersects the row $R$ of $L(\alpha)$ containing $u$ and $\overline{e}$ has positive slope.
Thus, the part of $P$ above $R$ is contained in the slice of positive slope determined by $\overline{e_2}$ and $\overline{e}$; see part~(b) of Figure~\ref{fig-altogether}.
By Lemma~\ref{lem-forbidden}, there is no edge of type I in $P$.
By Lemma~\ref{lem-typeII} and Corollary~\ref{cor-typeIV}, the total number of edges of~$P$ is then at most 
\[t_{II} + 1 + t_{IV} \leq 2\left\lceil\log_\alpha\left(\frac{\alpha}{\alpha-1}\right)\right\rceil+3.\]

In the rest of the proof, we can now assume that none of the edges $e_1$ and $e_2$ is vertical.
We can label them so that the slope of~$e_1$ is larger than the slope of $e_2$.

First, assume that the slope of~$e_1$ is positive and the slope of $e_2$ is negative.
Then, since the vertices of $P$ do not lie on two columns of $L(\alpha)$, the point $(\alpha \cdot x(u),y(u))$ is contained in the interior of $P$, which is impossible as $P$ is empty in $L(\alpha)$.

If the slopes of $e_1$ and $e_2$ are both non-positive, then there is no edge of type II  besides the possibly horizontal edge $e_1$ as $u$ is the leftmost vertex of $P$.
By Lemma~\ref{lem-forbidden}, there is also no non-vertical edge of type IV as $P$ is contained in the slice of negative slopes determined by $\overline{e_1}$ and $\overline{e_2}$ or by $\overline{e}$ and $\overline{e_2}$ if $e_1$ is horizontal; see part~(c) of Figure~\ref{fig-altogether}.
Thus, by Lemmas~\ref{lem-typeI} and~\ref{lem-typeIII}, the number of edges of $P$ is at most 
\[t_I + 1 + t_{III} + 1 \leq \left\lceil\log_\alpha\left(\frac{\alpha}{\alpha-1}\right)\right\rceil+2\left\lceil \log_{\alpha}\left(\frac{\alpha+1}{\alpha}\right) \right\rceil +3\]
for $\alpha \in (1,2)$. 
Moreover, for $\alpha \geq 2$ this bound may be reduced by one.

If the slopes of $e_1$ and $e_2$ are both non-negative, then there is no edge of type III besides the possibly horizontal edge $e_2$ (note that a vertical edge of type III would have $u$ as its bottom vertex,  which is impossible by the choice of $u$).
Then, $P$ is contained in the slice of positive slope determined by $\overline{e_1}$ and $\overline{e_2}$ or, if $e_2$ is horizontal, by $\overline{e_1}$ and $\overline{e'}$; see part~(d) of Figure~\ref{fig-altogether}.
Lemma~\ref{lem-forbidden} then implies that there is also no edge of type I.
We thus have at most
\[t_{II} + 1  + t_{IV}  \leq 2\left\lceil\log_\alpha\left(\frac{\alpha}{\alpha-1}\right)\right\rceil+3\]
edges of $P$ by Lemma~\ref{lem-typeII} and Corollary~\ref{cor-typeIV}.

Altogether, the upper bound on the number of edges of $P$ is \[\max\left\{\left\lceil\log_\alpha\left(\frac{\alpha}{\alpha-1}\right)\right\rceil+2\left\lceil \log_{\alpha}\left(\frac{\alpha+1}{\alpha}\right) \right\rceil+3, \;\; 2\left\lceil\log_\alpha\left(\frac{\alpha}{\alpha-1}\right)\right\rceil+3\right\}\]
for $\alpha \in (1,2)$.
Moreover, the first term may be reduced by one for $\alpha \geq 2$.
This becomes $5$ for $\alpha \geq 2$, $h(\alpha) \leq 7$ for $\alpha \geq [\frac{1+\sqrt{5}}{2},2)$, and at most $3\left\lceil\log_\alpha\left(\frac{\alpha}{\alpha-1}\right)\right\rceil+3$ otherwise, since $\left\lceil \log_{\alpha}\left(\frac{\alpha+1}{\alpha}\right) \right\rceil \leq \left\lceil\log_\alpha\left(\frac{\alpha}{\alpha-1}\right)\right\rceil$ for every $\alpha \in (1,\frac{1+\sqrt{5}}{2})$.
\end{proof}

\section{Proof of Theorem~\ref{thm-lowerBound}}
\label{sec-lowerBound}

We prove the lower bounds on $h(\alpha)$ through the following three propositions.

\begin{proposition}
\label{prop-lowerBound-2}
For every $\alpha\geq 2$, we have $h(\alpha)\geq 5$.
\end{proposition}

\begin{proof} It is easy to check that $\textrm{conv}\{(1,\alpha^2),(\alpha,\alpha),(\alpha^2,1),(\alpha^2,\alpha),(\alpha,\alpha^2)\}$ is an empty polygon in $L(\alpha)$ with $5$ vertices for any $\alpha$; see Figure~\ref{fig-Q}.
\end{proof}

\begin{proposition}
\label{prop-lowerBound-2-phi}
For every $\alpha\in [\frac{1+\sqrt{5}}{2},2)$, we have $h(\alpha)\geq 7$.
\end{proposition}

\begin{proof} Let $k=k(\alpha)$ be a sufficiently large integer, and let \[Q(\alpha)=\{(1,\alpha^k),(\alpha^{k-2},\alpha^{k-1}),(\alpha^{k-1},\alpha^{k-2}),(\alpha^k,1),(\alpha^k,\alpha),(\alpha^{k-1},\alpha^{k-1}),(\alpha,\alpha^k)\};\]
see Figure~\ref{fig-Q}.
We will show that $\textrm{conv}(Q(\alpha))$
is an empty polygon in $L(\alpha)$ with $7$ vertices.

\begin{figure}[ht]
    \centering
    \includegraphics{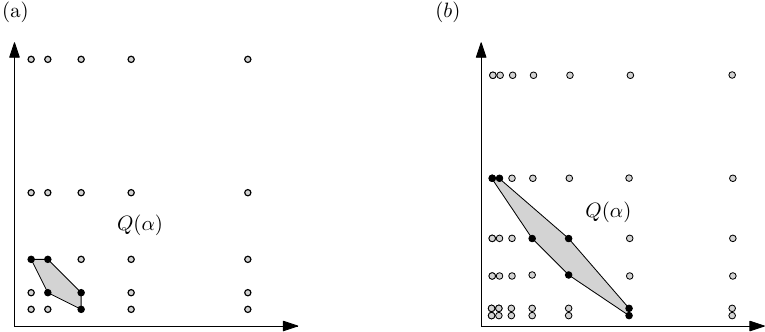}
    \caption{(a) An illustration of the proof of Proposition~\ref{prop-lowerBound-2}. (b) An illustration of the proof of Proposition~\ref{prop-lowerBound-2-phi}.}
    \label{fig-Q}
\end{figure}

First, we show that $Q(\alpha)\setminus \{(\alpha^{k-1},\alpha^{k-1})\}$ is in convex position. 
For this, by symmetry, it is enough to check that $\{(\alpha^{k-1}, \alpha^{k-2}), (\alpha^k, 1), (1, \alpha^k)\}$ is oriented counterclockwise. This is the case exactly if $\alpha^{k-1}-\alpha^k+\alpha^{k-2}-1<0$. By rearranging we get $\alpha^{k-2}(\alpha+1-\alpha^2)<1$, which holds for any $k$, since $\alpha+1-\alpha^2\leq 0$ as $\alpha \geq (1+\sqrt{5})/2$.

Now, to show that the set $Q(\alpha)$ is in convex position, it is sufficient to check that $\{(1, \alpha^k), (\alpha^k, \alpha), (\alpha^{k-1}, \alpha^{k-1})\}$ is oriented counterclockwise. This holds exactly if $\alpha^{k-1}-\alpha^k+\alpha^{k-1}-\alpha\geq 0.$ By rearranging we get $2\alpha^{k-2}(2-\alpha)\geq 1$. Since $1<\alpha <2$, this holds if $k$ is sufficiently large.

Thus, $\textrm{conv}(Q(\alpha))$ has $7$ vertices. 
To show that $\textrm{conv}(Q(\alpha))$ is empty in $L(\alpha)$, we remark that points of the exponential lattice $L(\alpha)$ with both coordinates smaller than $\alpha^{k-1}$ are below the line through $(\alpha^{k-1},\alpha^{k-2})$ and $(\alpha^{k-2},\alpha^{k-1})$. Further, points with at least one coordinate larger than $\alpha^{k-1}$ are either above the line through $(1,\alpha^k)$ and $(\alpha,\alpha^k)$ or to the right of the line through $(\alpha^k,1)$ and $(\alpha^k,\alpha)$.
\end{proof}

\begin{proposition}\label{prop-lowerBound-phi}
For every $\alpha > 1$, we have $h(\alpha)\geq \left\lfloor\sqrt{\frac{1}{\alpha-1}}\right\rfloor$.
\end{proposition}
\begin{proof}
For a positive integer $k$, let $P(k)=\{(\alpha^i,\alpha^{k-i}):1\leq i \leq k\}$. Since $P(k)$ is contained in the hyperbola $h= \{(x,y) \in \mathbb{R}^2\colon x,y>0, xy=\alpha^k\}$, the points of $P(k)$ are in convex position, and $\textrm{conv}(P(k))$ has $k$ vertices. We will show that if $k\leq \sqrt{\frac{1}{\alpha-1}}$, then $\textrm{conv}(P(k))$ is empty.

For points $(x,y)$ of $L(\alpha)$ above $h$, we have $xy\geq \alpha^{k+1}$. Further, points $(x,y)$ of $L(\alpha)$ with $xy\geq \alpha^{k+2}$ are separated from $h$ by the hyperbola $h'= \{(x,y) \in \mathbb{R}^2\colon x,y>0, xy=\alpha^{k+1}\}$. Thus, it is sufficient to check that $h'$ is above the line $\ell$ connecting $(1,\alpha^k)$ with $(\alpha^k,1)$. The closest point of $h'$ to $\ell$ is $(\alpha^{(k+1)/2},\alpha^{(k+1)/2})$, thus it is sufficient to check that this point is above $\ell$. This holds if $2\alpha^{(k+1)/2}-\alpha^k-1\geq 0$ and we show that this inequality is satisfied for $k\leq \sqrt{\frac{1}{\alpha-1}}$.

Let $\alpha = 1 + s^2$ with some $s \in (0,1)$. In this notation, $k \leq 1/s$ and we need to prove that $2(1+s^2)^{(k+1)/2}\geq (1+s^2)^{k}+1$. Since $(1+s^2)^{(k+1)/2} \geq 1 + s^2\frac{k+1}{2}$ by the Bernoulli inequality, and $(1+s^2)^{k} \leq e^{s^2k} $, it is sufficient to prove the stronger inequality $2(1 + s^2\frac{k+1}{2} )\geq e^{s^2k} + 1$. By taking the derivative of both sides with respect to $k$ we have $s^2 \leq s^2e^{s^2k}$, thus it is enough to check the inequality when $k$ is maximal. If $k = 1/s$, it is equivalent to $1 + s + s^2 \geq e^s$, which holds for $s \in (0,1)$ as $e^s = 1 + s + \sum_{n=2}^\infty \frac{s^n}{n!} \leq 1 + s + s^2 \sum_{n=2}^\infty \frac{1}{n!} = 1 + s + s^2(e-2)<1 + s + s^2$.
\end{proof}

\section{Proof of Proposition~\ref{prop-fibonacci}}
\label{sec-fibonacci}

Let us denote $\mathcal{F} = \{F_n \colon n \in \mathbb{N}_0\}^2$.
For every positive integer $k$, we show that $h(\mathcal{F}) \geq k+1$.
We will show that the points $(F_{i+2},F_i)$ with odd $i\in \{1,\dots,2k+1\}$ are vertices of an empty convex polygon $P$ in $\mathcal{F}$. 

First, we show that the points $(F_{i+2},F_i)$ with odd $i\in \{1,\dots,2k+1\}$ are in convex position.
To show that, it suffices to show that the slopes of lines determined by three consecutive such points are decreasing.
That is, we want to prove
\[\frac{F_i-F_{i-2}}{F_{i+2}-F_i} > \frac{F_{i+2}-F_i}{F_{i+4}-F_{i+2}}\]
for every odd $i\in \{1,\dots,2k-3\}$.
Since $F_k = F_{k-1}+F_{k-2}$ for every $k\geq 2$, this inequality can be rewritten as
\[\frac{F_{i-1}}{F_{i+1}} > \frac{F_{i+1}}{F_{i+3}}.\]
Thus, we want to show that $F_{i-1} \cdot F_{i+3} > F^2_{i+1}$ for odd $i$.
This is indeed true, as $F_{i-1} \cdot F_{i+3} - F^2_{i+1} = (-1)^{i+1-2}F_2^2 > 0$ by the Catalan identity of the Fibonacci numbers.

To show that the polygon $P$ is empty in $\mathcal{F}$, consider the line $L = \{(x,y) \in \mathbb{R}^2 \colon y= x/\varphi^2\}$.
Any point $(F_{i+2},F_i)$ with odd $i$ lies below $L$ because
\[\frac{F_{i+2}}{\varphi^2} = \frac{1}{\varphi^2} \cdot \frac{\varphi^{i+3} - \psi^{i+3} }{\sqrt{5}} > \frac{\varphi^{i+1} - \psi^{i+1}}{\sqrt{5}} = F_i\]
since $\varphi^2>\psi^2$ and $i+3$, $i+1$ are both even implying $\psi^{i+3},\psi^{i+1} > 0$.
Analogously, all points $(F_{i+2},F_i)$ with even $i$ lie above $L$.
For any $i$, every point $(F_j,F_i)$ with $j \leq i+1$ lies above $L$, because $F_i \geq F_{j-1} > F_j/\varphi^2$.
Each point $(F_{i+2},F_i)$ with odd $i$ lies at vertical distance less than $1/2$ from $L$ as
\begin{align*}\frac{F_{i+2}}{\varphi^2} &= \frac{1}{\varphi^2} \cdot \frac{\varphi^{i+3} - \psi^{i+3} }{\sqrt{5}} = \frac{\varphi^{i+1} - \psi^{i+1}}{\sqrt{5}} + \frac{\varphi^2\psi^{i+1}-\psi^{i+3}}{\varphi^2\sqrt{5}} \leq F_i + \frac{\varphi^2\psi^2-\psi^4}{\sqrt{5}} \\
&< F_i + \frac{1}{2}.
\end{align*}
Any point $(F_{i+2},F_j)$ with $j \leq i-1$ lies below $L$ at vertical distance at least $1/2$ since the distance is either at least $F_i-F_j \geq 1$ if $i$ is odd or it is at least $F_i-F_j-\frac{1}{2}\geq\frac{1}{2}$ if $i$ is even.
Thus the only points of $\mathcal{F}$ lying between the parallel lines $y = x/\varphi^2 - 1/2$ and $L$ are the points $(F_{i+2}, F_i)$ with $i$ odd.
It follows that $P$ is an empty convex polygon in $\mathcal{F}$ and $h(\mathcal{F}) \geq k+1$.

\section{Proof of Theorem~\ref{thm-nondiagonal}}
\label{sec-nondiagonal}

Let $\alpha,\beta>1$ be two real numbers. 
We prove that $h(L(\alpha,\beta))$ is finite if and only if $\log_\alpha(\beta)$ is a rational number.

\subsection{Finite upper bound}

First, assume that $\log_\alpha(\beta) \in \mathbb{Q}$.
We will use Theorem~\ref{thm-upperBound} to show that the number $h(L(\alpha,\beta))$ is finite.
Since $\log_\alpha(\beta) \in \mathbb{Q}$ and $\alpha,\beta>1$, there are positive integers $p$  and $q$ such that $\beta = \alpha^{p/q}$.
Suppose for contradiction that there is an empty polygon $P$ in $L(\alpha,\beta)$ with at least $pq\cdot h(\alpha^p)+1$ vertices. 
Note that this number of vertices is finite by Theorem~\ref{thm-upperBound}.
For $k \in \{0,\dots,q-1\}$, we call a row of $L(\alpha,\beta)$ \emph{congruent to $k$} if it is of the form $\{\alpha^n \colon n \in \mathbb{N}_0\} \times \beta^m$  for some integer $m$ congruent to $k$ modulo~$q$. 
Analogously, a column of $L(\alpha,\beta)$ is \emph{congruent to $\ell \in\{0,\dots,p-1\}$} if it is of the form $\alpha^m \times \{\beta^n \colon n \in \mathbb{N}_0\}$ for some $m$ congruent to $\ell$ modulo $p$.

Now, since $P$ contains at least $pq\cdot h(\alpha^p)+1$ vertices, the pigeonhole principle implies that there are integers $k\in \{0,\dots,q-1\}$ and $\ell \in \{0,\dots,p-1\}$ such that at least $h(\alpha^p)+1$ vertices of $P$ that all lie in rows congruent to $k$ and in columns congruent to~$\ell$.
Let $P'$ be the convex polygon that is spanned by these vertices.
We claim that the polygon $P'$ is not empty in~$L(\alpha,\beta)$.
Since $P' \subseteq P$, we get that $P$ is also not empty in $L(\alpha,\beta)$, which contradicts our assumption about $P$.

To show that $P'$ is not empty in $L(\alpha,\beta)$, consider the subset $L$ of $L(\alpha,\beta)$ that contains only points of $L(\alpha,\beta)$ that lie in rows congruent to $k$ and in columns congruent to $\ell$.
Clearly, vertices of $P'$ lie in $L$ and $L$ is an affine image of $L(\alpha^p)$, which is scaled by the factors $\alpha^\ell$ and $\beta^k=\alpha^{kp/q}$ in the $x$- and $y$-direction, respectively.
Since affine mappings preserve incidences and $P'$ has at least $h(\alpha^p)+1$ vertices, it follows that $P'$ is not empty in~$L$.
Since $L \subseteq L(\alpha,\beta)$, $P'$ is not empty in $L(\alpha,\beta)$ either.

\subsection{Finite lower bound}

Let $\log_\alpha (\beta) \in \mathbb{Q}$ and $\beta = \alpha^{p/q}$ for some relative prime positive integers $p$ and $q$.
Observe that in this case $L(\alpha, \beta) \subset L(\alpha^{1/q})$. Thus, if an empty polygon in $L(\alpha^{1/q})$ is a subset of $L(\alpha, \beta)$, then it is an empty polygon in $L(\alpha, \beta)$.

Let $k = \left\lfloor\sqrt{1/(\alpha^{1/q} - 1)}\right\rfloor$ and consider the set $P= \{(\alpha^{i/q}, \alpha^{(k-i)/q}): 1 \leq i \leq k\}$. It is an empty polygon in $L(\alpha^{1/q})$, as it is shown in the proof of Proposition~\ref{prop-lowerBound-phi}. Since its subset $P ' = \{(\alpha^{i/q}, \alpha^{(k-i)/q}): 1 \leq i \leq k \text{ with } q | i \text{ and } p | k -i \}$ is a subset of $L(\alpha, \beta)$ and an empty polygon in $L(\alpha^{1/q})$, it is an empty polygon in $L(\alpha, \beta)$ with $\lfloor k/pq \rfloor$ vertices.

\subsection{Infinite lower bound}

Now, assume that $\log_\alpha(\beta) \notin \mathbb{Q}$.
We will find a subset of $L(\alpha,\beta)$ forming empty convex polygon in $L(\alpha,\beta)$ with arbitrarily many vertices.
To do so, we use the theory of continued fractions, so we first introduce some definitions and notation.

\subsubsection{Continued fractions}

Here, we recall mostly basic facts about continued fractions, which we use in the proof.
Most of the results that we state can be found, for example, in the book by Khinchin~\cite{khin97}.

For a positive real number $r$, the \emph{(simple) continued fraction of $r$} is an expression of the form 
\[r=a_0 + \frac{1}{a_1+\frac{1}{a_2 + \frac{1}{a_3 + \cdots}}},\]
where $a_0 \in \mathbb{N}_0$ and $a_1,a_2,\dots$ are positive integers.
The simple continued fraction of $r$ can be written in a compact notation as 
\[[a_0;a_1,a_2,a_3,\dots].\]
For every $n \in \mathbb{N}_0$, if we denote $\frac{p_n}{q_n} = [a_0;a_1,a_2,\dots,a_n]$ and set $p_{-1}=1$, $p_0=a_0$, $q_{-1}=0$, $q_0=1$, then the numbers $p_n$ and $q_n$ satisfy the recurrence
\begin{equation}
\label{eq-continued}
p_n = a_np_{n-1} + p_{n-2} \;\;\text{ and }\;\;
q_n = a_nq_{n-1} + q_{n-2}
\end{equation}
for each $n\in\mathbb{N}$.
Observe that if $r$ is irrational, then its continued fraction has infinitely many coefficients.
Also, it follows from~\eqref{eq-continued} that $\frac{p_n}{q_n} < r$ for $n$ even and $\frac{p_n}{q_n} > r$ for $n$ odd.

For example, if $r=\log_2(3)$, we get the continued fraction $[1;1,1,2,2,3,1,5,2,23,\dots]$ and the sequence $\left(\frac{p_n}{q_n}\right)_{n \in \mathbb{N}_0} = \left(\frac{1}{1},\frac{2}{1},\frac{3}{2},\frac{8}{5},\frac{19}{12},\frac{65}{41},\frac{84}{53},\frac{485}{306},\dots\right)$.
For $r=\frac{1+\sqrt{5}}{2}$, we have $[1;1,1,1,\dots]$ and $\left(\frac{p_n}{q_n}\right)_{n \in \mathbb{N}_0} = \left(\frac{1}{1},\frac{2}{1},\frac{3}{2},\frac{5}{3},\frac{8}{5},\frac{13}{8},\frac{21}{13},\frac{34}{21},\dots\right)$.

We will call the fractions $\frac{p_n}{q_n}$ the \emph{convergents} of $r$. A \emph{semi-convergent} of $r$ is a number
$\frac{p_{n-1}+ip_n}{q_{n-1}+iq_n}$ where $i \in \{0,1,\dots,a_{n+1}\}$. Note that each convergent of $r$ is also a semi-convergent of $r$. The names are motivated by the use of convergents and semi-convergents as rational approximations of an irrational number $r$.

A rational number $\frac{p}{q}$ is a \emph{best approximation} of an irrational number $r$, if any fraction $\frac{p'}{q'} \neq \frac{p}{q}$ with $q' < q$ satisfies \[\left|q'\left(r-\frac{p'}{q'}\right)\right| > \left| q\left(r-\frac{p}{q}\right) \right|.\]
A rational number $\frac{p}{q}$ is a \emph{best lower approximation of $r$} if 
\[q'\left(r-\frac{p'}{q'}\right) > q\left(r-\frac{p}{q}\right) \geq 0\]
for all rational numbers $\frac{p'}{q'}$ with $\frac{p'}{q'} \leq r$, $\frac{p}{q} \neq \frac{p'}{q'}$, and $0 < q' \leq q$.
Similarly,  
$\frac{p}{q}$ is a \emph{best upper approximation of $r$} if 
\[q'\left(r-\frac{p'}{q'}\right) < q\left(r-\frac{p}{q}\right) \leq 0\]
for all rational numbers $\frac{p'}{q'}$ with $\frac{p'}{q'} \geq r$, $\frac{p}{q} \neq \frac{p'}{q'}$, and $0 < q' \leq q$. 

\iffalse
If $r$ is irrational, the \emph{convergents} $\frac{p_n}{q_n}$,  of $r$ are the \emph{best  approximations} of $r$, that is, any fraction $\frac{p}{q}$ with $q < q_n$ satisfies 
\[
|r q - p| > |r q_{n-1} - p_{n-1}|.
\]
In other words, $\frac{p_n}{q_n}$ is closer to $r$ than any approximation with a smaller or equal denominator.

We will use a result about one-sided approximations of $r$.
To state it, we need to introduce a few definitions.
A \emph{semi-convergent} of $r$ is a number
$\frac{p_{n-1}+ip_n}{q_{n-1}+iq_n}$ where $i \in \{0,1,\dots,a_{n+1}-1\}$.
Note that each convergent of $r$ is also a semi-convergent of $r$.
A rational number $\frac{p}{q}$ is a \emph{best lower approximation of $r$} if 
\[0 \leq q\left(r-\frac{p}{q}\right) < q'\left(r-\frac{p'}{q'}\right)\]
for all rational numbers $\frac{p'}{q'}$ with $\frac{p'}{q'} \leq r$, $\frac{p}{q} \neq \frac{p'}{q'}$, and $0 < q' \leq q$.
Similarly,  
$\frac{p}{q}$ is a \emph{best upper approximation of $r$} if 
\[0 \leq q\left(\frac{p}{q}-r\right) < q'\left(\frac{p'}{q'}-r\right)\]
for all rational numbers $\frac{p'}{q'}$ with $\frac{p'}{q'} \geq r$, $\frac{p}{q} \neq \frac{p'}{q'}$, and $0 < q' \leq q$.

\fi

It is a well known fact that convergents are best approximations of $r$~\cite{khin97}.
The following lemma about best lower and best upper approximations is a recent result of Han{\v{c}}l and Turek \cite{hanTur19}. Our definitions of best lower or upper approximations correspond to their definitions of best lower or upper approximations \emph{of the second kind}. The lemma follows from Theorem 4.5 of \cite{hanTur19}.

\begin{lemma}[\cite{hanTur19}]
\label{lem-bestApprox}
Let $r$ be a real number with $r=[a_0;a_1,a_2,\dots]$ and let $\frac{p_n}{q_n}$ be the $n$th convergent of $r$ for each $n \in \mathbb{N}_0$.
Then, the following two statements hold.
\begin{enumerate}
    \item\label{item2} The set of best lower approximations of $r$ consists of semi-convergents $\frac{p_{n-1}+ip_n}{q_{n-1}+iq_n}$ of~$r$ with $n$ odd and $0 \leq i < a_{n+1}$.
     \item\label{item3} The set of best upper approximations of $r$ consists of semi-convergents $\frac{p_{n-1}+ip_n}{q_{n-1}+iq_n}$ of~$r$ with $n$ even and $0 \leq i < a_{n+1}$, except for the pair $(n,i)=(0,0)$.
\end{enumerate}

\end{lemma}

Finally, a real number $r$ is \emph{restricted} if there is a positive integer $M$ such that all the partial denominators $a_i$ from the continued fraction of $r$ are at most $M$.
The restricted numbers are exactly those numbers $r$ that are badly approximable by rationals~\cite{khin97}, that is, there is a constant $c>0$ such that for every $\frac{p}{q} \in \mathbb{Q}$ we have $\left|r-\frac{p}{q}\right|>\frac{c}{q^2}$.

We divide the rest of the proof of Theorem~\ref{thm-nondiagonal} into two cases, depending on whether $\log_\alpha(\beta)$ is restricted or not.

\subsubsection{Unrestricted case}

First, we  assume that $\log_\alpha(\beta)$ is not restricted.
Let $[a_0;a_1,a_2,a_3,\dots]$ be the continued fraction of $\log_\alpha(\beta)$ with $\frac{p_n}{q_n} = [a_0;a_1,\dots,a_n]$ for every $n \in \mathbb{N}_0$.
Then, for every positive integer $m$, there is a positive integer $n(m)$ such that $a_{n(m)+1} \geq m$. 
We use this assumption to construct, for every positive integer $m$, a convex polygon with at least $m$ vertices from $L(\alpha,\beta)$ that is empty in $L(\alpha,\beta)$.

For a given $m$, consider the integer $n(m)$ and let $W$ be the set of points
\[w_i = (\alpha^{p_{n(m)-1}+ip_{n(m)}},\beta^{q_{n(m)-1}+iq_{n(m)}})\]
where $i \in \{0,1,\dots,a_{n(m)+1}\}$.
That is, we consider points where the exponents form semi-convergents $\frac{p_{n(m)-1}+ip_{n(m)}}{q_{n(m)-1}+iq_{n(m)}}$ to $\log_\alpha(\beta)$.
We abbreviate $p_{n,i} = p_{n(m)-1}+ip_{n(m)}$ and $q_{n,i} = q_{n(m)-1}+iq_{n(m)}$.
Observe that $|W|\geq m$.
We will show that $W$ is the vertex set of an empty convex polygon in $L(\alpha,\beta)$.
To do so, we assume without loss of generality that $n(m)$ is even so that $\frac{\beta^{q_{n(m)}}}{\alpha^{p_{n(m)}}} > 1$.
The other case when $n(m)$ is odd is analogous.

First, we show that $W$ is in convex position.
In fact, we prove that all triples $(w_{i_1},w_{i_2},w_{i_3})$ with $i_1<i_2<i_3$ are oriented counterclockwise.
It suffices to show this for every triple $(w_i,w_{i+1},w_{i+2})$.
To do so, we need to prove the inequality
\[\frac{y(w_{i+2})-y(w_{i+1})}{x(w_{i+2})-x(w_{i+1})} = \frac{\beta^{q_{n,i+2}}-\beta^{q_{n,i+1}}}{\alpha^{p_{n,i+2}}-\alpha^{p_{n,i+1}}} > \frac{\beta^{q_{n,i+1}}-\beta^{q_{n,i}}}{\alpha^{p_{n,i+1}}-\alpha^{p_{n,i}}} = \frac{y(w_{i+1})-y(w_i)}{x(w_{i+1})-x(w_i)}.\]
After dividing by $\frac{\beta^{q_{n(m)-1}}}{\alpha^{p_{n(m)-1}}}$, this can be written as
\[\frac{\beta^{(i+2)q_{n(m)}}-\beta^{(i+1)q_{n(m)}}}{\alpha^{(i+2)p_{n(m)}}-\alpha^{(i+1)p_{n(m)}}} > \frac{\beta^{(i+1)q_{n(m)}}-\beta^{iq^{n(m)}}}{\alpha^{(i+1)p_{n(m)}}-\alpha^{ip_{n(m)}}}.\]
If divide both sides by $\frac{\beta^{(i+1)q_{n(m)}}-\beta^{iq_{n(m)}}}{\alpha^{(i+1)p_{n(m)}}-\alpha^{ip_{n(m)}}}$, then the above inequality becomes \[\frac{\beta^{q_{n(m)}}}{\alpha^{p_{n(m)}}} > 1.\]
This is true as ${n(m)}$ is even.

It remains to prove that the polygon $Q$ with the vertex set $W$ is empty in $L(\alpha,\beta)$.
Suppose for contradiction that there is a point $(\alpha^p,\beta^q)$ of $L(\alpha,\beta)$ lying in the interior of~$Q$.
Let $i$ be the minimum positive integer from $\{1,\dots,a_{n(m)+1}\}$ such that $q<q_{n,i}$.
Such an $i$ exists as $(\alpha^p,\beta^q)$ is in the interior of $Q$.
We then have $q_{n,i-1} < q< q_{n,i}$.
Since $(\alpha^p,\beta^q)$ is in the interior of $Q$ and $W$ lies below the line $x=y$, we have $\frac{p}{q} > \log_\alpha(\beta)$. So it is enough to prove that $(\alpha^p,\beta^q)$ does not lie above the line $\overline{w_{i-1}w_i}$.

We have $p_{n,i}-\log_\alpha(\beta)q_{n,i} < p_{n,i-1} - \log_\alpha(\beta)q_{n,i-1}$ as $\frac{p_{n,i}}{q_{n,i}}$ is a best upper approximation of $\log_\alpha(\beta)$ and $q_{n,i-1} < q_{n,i}$. This implies $\frac{\beta^{q_{n,i-1}}}{\alpha^{p_{n,i-1}}}<\frac{\beta^{q_{n,i}}}{\alpha^{p_{n,i}}}$, or equivalently that $w_i$ lies above the line determined by $w_{i-1}$ and the origin.

Now if $(\alpha^p, \beta^q)$ lies above the line $\overline{w_{i-1}w_i}$, then it also lies above the line determined by $w_{i-1}$ and the origin. Thus, $\frac{\beta^{q_{n,i-1}}}{\alpha^{p_{n,i-1}}}<\frac{\beta^q}{\alpha^p}$, implying
\[p-\log_\alpha(\beta)q < p_{n,i-1} - \log_\alpha(\beta)q_{n,i-1},\]
which means that $\frac{p}{q}$ is a better upper approximation of $\log_\alpha(\beta)$ than $\frac{p_{n,i-1}}{q_{n,i-1}}$. Thus, there exists a best upper approximation $\frac{p^*}{q^*}$ of $\log_\alpha(\beta)$ with $q_{n,i-1} < q^* < q_{n,i}$. This contradicts part~\ref{item3} of Lemma~\ref{lem-bestApprox} as $\frac{p^*}{q^*}$ is not a semi-convergent of $\log_\alpha(\beta)$.

\subsubsection{Restricted case}

Now, assume that the number $\log_\alpha(\beta)$ is restricted.
Let $[a_0;a_1,a_2,a_3,\dots]$ be the continued fraction of $\log_\alpha(\beta)$ with $\frac{p_n}{q_n} = [a_0;a_1,\dots,a_n]$ for every $n \in \mathbb{N}_0$.
Let $M=M(\alpha,\beta)$ be a number satisfying \begin{equation}\label{eq: M}
a_n \leq M
\end{equation} for every $n \in \mathbb{N}_0$ and let $c=c(\alpha,\beta)>0$ be a constant such that 
\begin{equation}\label{eq: c}
\left|\log_\alpha(\beta)-\frac{p}{q}\right|>\frac{c}{q^2}
\end{equation}
holds for every $\frac{p}{q}\in \mathbb{Q}$.
Recall that $\frac{\alpha^{p_n}}{\beta^{q_n}}<1$ for even $n$ and $\frac{\alpha^{p_n}}{\beta^{q_n}}>1$ for odd $n$. 
Note also that the sequence $\left(\frac{\alpha^{p_n}}{\beta^{q_n}}\right)_{n \in \mathbb{N}_0}$ converges to $1$ as  $\left(\frac{p_n}{q_n}\right)_{n \in \mathbb{N}_0}$ converges to $\log_\alpha(\beta)$.
Moreover, the terms of $\left(\frac{p_n}{q_n}\right)_{n \in \mathbb{N}_0}$ with odd indices form a decreasing subsequence and the terms with even indices determine an increasing subsequence.

Let $n_0=n_0(\alpha,\beta)$ be a sufficiently large positive integer and let $V$ be the set of points $v_n =(\alpha^{p_n},\beta^{q_n})$ for every odd $n \geq n_0$.
Note that $V$ is a subset of $L(\alpha,\beta)$.

We first show that $V$ is in convex position. 
In fact, we prove a stronger claim by showing that the orientation of every triple $(v_{n_1},v_{n_2},v_{n_3})$ with $n_1 < n_2 < n_3$ is counterclockwise.
It suffices to show this for every triple $(v_{n-4},v_{n-2},v_n)$.
To do so, we prove that the slopes of the lines determined by consecutive points of $V$ are increasing, that is,
\[\frac{y(v_n)-y(v_{n-2})}{x(v_n)-x(v_{n-2})} = \frac{\beta^{q_n}-\beta^{q_{n-2}}}{\alpha^{p_n}-\alpha^{p_{n-2}}}  > \frac{\beta^{q_{n-2}}-\beta^{q_{n-4}}}{\alpha^{p_{n-2}}-\alpha^{p_{n-4}}} = \frac{y(v_{n-2})-y(v_{n-4})}{x(v_{n-2})-x(v_{n-4})}\]
for every even $n \geq n_0$.
By dividing both sides of the inequality with $\frac{\beta^{q_{n-2}}}{\alpha^{p_{n-2}}}$, we rewrite this expression as
\[\frac{\beta^{q_n-q_{n-2}}-1}{\alpha^{p_n-p_{n-2}}-1}  > \frac{1-\beta^{q_{n-4}-q_{n-2}}}{1-\alpha^{p_{n-4}-p_{n-2}}}.\]
Using~\eqref{eq-continued}, this is the same as
\[\frac{\beta^{a_nq_{n-1}}-1}{\alpha^{a_np_{n-1}}-1}  > \frac{1-\beta^{-a_{n-2}q_{n-3}}}{1-\alpha^{-a_{n-2}p_{n-3}}}.\]
The above inequality can be rewritten as 
\[(\beta^{a_nq_{n-1}}-1)(1-\alpha^{-a_{n-2}p_{n-3}})>(\alpha^{a_np_{n-1}}-1)(1-\beta^{-a_{n-2}q_{n-3}}),\]
where $\beta^{q_{n-1}} > \alpha^{p_{n-1}} > 1$ as $n-1$ is even.
Therefore, if the above inequality holds for $a_n=1$, then it holds for any $a_n$ as this number is always at least $1$.
Thus, it suffices to show 
\begin{equation}
\label{eq-innequality}
(\beta^{q_{n-1}}-1)(1-\alpha^{-a_{n-2}p_{n-3}})>(\alpha^{p_{n-1}}-1)(1-\beta^{-a_{n-2}q_{n-3}}).
\end{equation}
We prove this using the following simple auxiliary lemma.

\begin{lemma}
\label{lem-function}
Consider the function $f \colon \mathbb{R}^+\times \mathbb{R}^+ \to \mathbb{R}$ given by  $f(x,y) = (x-1)(1-1/y)$.
Let $x,y,x',y'>1$ be real numbers such that 
$1-\frac{1}{y}-\frac{x}{x'}>0$.
Then, $f(x',y)>f(x,y')$.
\end{lemma}
\begin{proof}
We have
\begin{align*}
f(x',y) - f(x,y') &= \left(x'-1\right)\left(1-\frac{1}{y}\right) - \left(x-1\right)\left(1-\frac{1}{y'}\right)\\
&=x'-\frac{x'-1}{y}-x+\frac{x-1}{y'}>x'-\frac{x'}{y}-x=x'\left(1-\frac{1}{y}-\frac{x}{x'}\right)>0,
\end{align*}
where 
the last inequality follows from $1-\frac{1}{y}-\frac{x}{x'}>0$.
\end{proof}

Now, by choosing $x=\alpha^{p_{n-1}}$, $x'=\beta^{q_{n-1}}$, $y=\alpha^{a_{n-2}p_{n-3}}$, and $y'=\beta^{a_{n-2}q_{n-3}}$, the inequality~\eqref{eq-innequality} becomes $f(x',y)>f(x,y')$.
In order to prove it, we just need to verify the assumptions of Lemma~\ref{lem-function}.
We clearly have $x,x',y,y'>1$.
It now suffices to show $1-\frac{1}{y}-\frac{x}{x'}>0$.
By~\eqref{eq: c}, we obtain that $q_{n-1}\log_\alpha(\beta)-p_{n-1} \geq c/q_{n-1}$, thus 
\[\frac{x}{x'} =\frac{\alpha^{p_{n-1}}}{\beta^{q_{n-1}}} \leq \alpha^{-c/q_{n-1}}.\] 
Now, to bound $q_{n-1}$ in terms of $p_{n-3}$, equation~\eqref{eq-continued} gives
\begin{align*}
q_{n-1} &= a_{n-1}q_{n-2}+q_{n-3} \leq (M+1)q_{n-2} = (M+1)(a_{n-2}q_{n-3}+q_{n-4})\\
&\leq (M+1)^2q_{n-3} \leq 2\log_\beta(\alpha)(M+1)^2p_{n-3},
\end{align*}
where we used \eqref{eq: M} and $q_{n-4} \leq q_{n-3} \leq q_{n-2}$, $q_{n-3} \leq 2\log_\beta(\alpha)p_{n-3}$ for $n$ large enough.
It follows that $q_{n-1} \leq M'p_{n-3}$ for a suitable constant $M'=M'(\alpha,\beta)>0$.
Thus, 
\[1-\frac{1}{y}-\frac{x}{x'} \geq 1  - \alpha^{-a_{n-2}p_{n-3}}- \alpha^{-c/q_{n-1}} \geq 1  - \alpha^{-a_{n-2}p_{n-3}}- \alpha^{-c/(M'p_{n-3})},\]
which is at least 
\[\frac{c\ln{\alpha}}{2M'
p_{n-3}} - \frac{1}{\alpha^{a_{n-2}p_{n-3}}}\]
as $1-c\ln{\alpha}/(2M'
p_{n-3}) \geq e^{-2c\ln{\alpha}/(2M'
p_{n-3})} = \alpha^{-c/(M'p_{n-3})}$ if $0<c\ln{\alpha}/(2M'
p_{n-3})<1/2$.
The last expression is positive if $n \geq n_0$ and $n_0$ is sufficiently large so that $p_{n-3}$ is large enough.

It remains to show that the convex polygon $P$ with the vertex set $V$ is empty in~$L(\alpha,\beta)$.
We proceed analogously as in the unrestricted case.
Suppose for contradiction that there is a point $(\alpha^p,\beta^q)$ of $L(\alpha,\beta)$ lying in the interior of~$P$.
Then, let $v_n=(\alpha^{p_n},\beta^{q_n})$ be the lowest vertex of $P$ that has $(\alpha^p,\beta^q)$ below.
Such a vertex $v_n$ exists, as $V$ contains points with arbitrarily large $y$-coordinate.
By the choice of $v_n$, we obtain $q_{n-2} < q< q_n$.
Since $(\alpha^p,\beta^q)$ is in the interior of $P$ and $V$ lies below the line $x=y$, we have $\frac{p}{q} > \log_\alpha(\beta)>\frac{p_{n-1}}{q_{n-1}}$.
Moreover, since all triples from $V$ are oriented counterclockwise, the point $(\alpha^p,\beta^q)$ lies above the line $\overline{v_{n-2}v_n}$.

Let \[w_i = (\alpha^{p_{n-2}+ip_{n-1}},\beta^{q_{n-2}+iq_{n-1}})\]
where $i \in \{0,1,\dots,a_n\}$ similarly as in the proof of the unrestricted case.
There, it was shown that all the triples $w_{i-1},w_i, w_{i+1}$ are oriented counterclockwise, thus all the points $w_i$ with $i  \in\{1,\dots,a_n-1\}$ lie below the line $\overline{v_{n-2}v_n}$.
Thus, if $(\alpha^p,\beta^q)$ lies above the segment connecting $v_{n-2}$ and $v_n$, then there is an $i$ such that $(\alpha^p,\beta^q)$ lies above the segment connecting $w_{i-1}$ and $w_i$. As in the last two paragraphs of the proof of the unrestricted case, the position of $(\alpha^p, \beta^q)$ implies the inequality $p-\log_\alpha(\beta)q < p_{n,i-1} - \log_\alpha(\beta)q_{n,i-1}$, and the contradiction follows from part~\ref{item3} of Lemma~\ref{lem-bestApprox}, as there can be no best upper approximation of $\log_\alpha(\beta)$ which is not a semi-convergent of $\log_\alpha(\beta)$.

\iffalse
Since $\frac{\alpha^{p_{n-2}}}{\beta^{q_{n-2}}} > \frac{\alpha^{p_n}}{\beta^{q_n}}$, the point $(\alpha^p,\beta^q)$ consequently lies above the line determined by the origin and by $v_{n-2}$.
Thus, $\frac{\beta^q}{\alpha^p}>\frac{\beta^{q_{n-2}}}{\alpha^{p_{n-2}}}$, implying
\[p-\log_\alpha(\beta)q < p_{n-2} - \log_\alpha(\beta)q_{n-2}\]
as $p>\log_\alpha(\beta)q$ and $p_{n-2} > \log_\alpha(\beta)q_{n-2}$.

Assume first that
$\frac{\beta^q}{\alpha^p}<\frac{\beta^{q_n}}{\alpha^{p_n}}$.
We set $x=\beta^{q-q_{n-2}}$, $y=\alpha^{p-p_{n-2}}$, $a=\beta^{q_n-q_{n-2}}$, and $b=\alpha^{p_n-p_{n-2}}$.
Then, we have $\frac{x}{y}<\frac{a}{b}$ and $1<x<a$ as $q<q_n$.
This gives $\frac{x-1}{y-1}<\frac{a-1}{b-1}$. 
In other words, we have $\frac{\beta^q-\beta^{q_{n-2}}}{\alpha^p-\alpha^{p_{n-2}}} < \frac{\beta^{q_n}-\beta^{q_{n-2}}}{\alpha^{p_n}-\alpha^{p_{n-2}}}$ implying that the point $(\alpha^p,\beta^q)$ lies below the line determined by $v_n$ and $v_{n-2}$.
This is impossible by our earlier observation.

Thus, $\frac{\beta^q}{\alpha^p}>\frac{\beta^{q_n}}{\alpha^{p_n}}$, which gives
\[p-\log_\alpha(\beta)q < p_n - \log_\alpha(\beta)q_n.\]
Thus, $(p,q)$ is a best approximation of $r$
However, $\frac{p}{q}$ is not a convergent of $r$ as $q_{n-2} < q < q_n$ and $\frac{p}{q} > \frac{p_{n-1}}{q_{n-1}}$.
This contradicts part~\ref{item1} of Lemma~\ref{lem-bestApprox}.

\fi

\paragraph{Acknowledgment}
This research was initiated at the 11th Eml\'{e}kt\'{a}bla workshop on combinatorics  and geometry.
We would like to thank G\'{e}za T\'{o}th for interesting discussions about the problem during the early stages of the research.

\bibliography{mybibliography}
\bibliographystyle{plain}

\end{document}